\documentclass[11pt,oneside,reqno]{article}

\usepackage{amsmath,amsthm,amssymb,color,bm,graphicx}

\usepackage{array}
\usepackage[margin=2.5cm,a4paper]{geometry}
\usepackage{bbm}
\usepackage{paralist}
\usepackage{mathrsfs}
\usepackage[breaklinks=true]{hyperref}
\hypersetup{
colorlinks=true,
linkcolor=black,
urlcolor=blue,
citecolor=black
}

\usepackage[square]{natbib} 

\renewcommand{\cite}{\citep*}

\numberwithin{equation}{section}
\allowdisplaybreaks[4]

\theoremstyle{plain}
\newtheorem{theorem}{Theorem}[section]

\newtheorem{corollary}[theorem]{Corollary}

\theoremstyle{definition}
\newtheorem{definition}[theorem]{Definition}

\makeatletter

\renewcommand{\phi}{\varphi}
\newcommand{\eps}{\varepsilon}
\newcommand{\eq}{\eqref}

\newcommand{\bigo}{\mathrm{O}}
\newcommand{\lito}{\mathrm{o}}

\newcommand{\ind}{\mathbf{1}}

\newcommand{\U}{\mathrm{U}}
\def\E{\mathbbm{E}}

\newcommand{\law}{\mathscr{L}}

\newcounter{ctr}\loop\stepcounter{ctr}\edef\X{\@Alph\c@ctr}%
	\expandafter\edef\csname s\X\endcsname{\noexpand\mathscr{\X}}
	\expandafter\edef\csname c\X\endcsname{\noexpand\mathcal{\X}}
	\expandafter\edef\csname b\X\endcsname{\noexpand\boldsymbol{\X}}
	\expandafter\edef\csname I\X\endcsname{\noexpand\mathbbm{\X}}
	\expandafter\edef\csname r\X\endcsname{\noexpand\mathrm{\X}}
\ifnum\thectr<26\repeat

\def\ben#1{\begin{equation}#1\end{equation}}

\def\ba#1{\begin{align*}#1\end{align*}}
\def\ban#1{\begin{align}#1\end{align}}

\def\given{\typeout{Command 'given' should only be used within bracket command}}
\newcounter{@bracketlevel}
\def\@bracketfactory#1#2#3#4#5#6{
\expandafter\def\csname#1\endcsname##1{%
\addtocounter{@bracketlevel}{1}%
\global\expandafter\let\csname @middummy\alph{@bracketlevel}\endcsname\given%
\global\def\given{\mskip#5\csname#4\endcsname\vert\mskip#6}\csname#4l\endcsname#2##1\csname#4r\endcsname#3%
\global\expandafter\let\expandafter\given\csname @middummy\alph{@bracketlevel}\endcsname
\addtocounter{@bracketlevel}{-1}}%
}
\def\bracketfactory#1#2#3{%
\@bracketfactory{#1}{#2}{#3}{relax}{1mu plus 0.25mu minus 0.25mu}{0.6mu plus 0.15mu minus 0.15mu}
\@bracketfactory{b#1}{#2}{#3}{big}{1mu plus 0.25mu minus 0.25mu}{0.6mu plus 0.15mu minus 0.15mu}
\@bracketfactory{bb#1}{#2}{#3}{Big}{2.4mu plus 0.8mu minus 0.8mu}{1.8mu plus 0.6mu minus 0.6mu}
\@bracketfactory{bbb#1}{#2}{#3}{bigg}{3.2mu plus 1mu minus 1mu}{2.4mu plus 0.75mu minus 0.75mu}
\@bracketfactory{bbbb#1}{#2}{#3}{Bigg}{4mu plus 1mu minus 1mu}{3mu plus 0.75mu minus 0.75mu}
}
\bracketfactory{clc}{\lbrace}{\rbrace}
\bracketfactory{clr}{(}{)}
\bracketfactory{cls}{[}{]}

\def\abs#1{\vert#1\vert}

\def\bbbabs#1{\biggl\vert#1\biggr\vert}

\renewcommand\section{\@startsection {section}{1}{\z@}%
{-3.5ex \@plus -1ex \@minus -.2ex}%
{1.3ex \@plus.2ex}%
{\center\small\sc\mathversion{bold}\MakeUppercase}}

\def\subsection#1{\@startsection {subsection}{2}{0pt}%
{-3.5ex \@plus -1ex \@minus -.2ex}%
{1ex \@plus.2ex}%
{\bf\mathversion{bold}}{#1}}

\def\subsubsection#1{\@startsection{subsubsection}{3}{0pt}%
{\medskipamount}%
{-10pt}%
{\normalsize\itshape}{\kern-2.2ex. #1.}}

\def\blfootnote{\xdef\@thefnmark{}\@footnotetext}

\makeatother

\newcommand{\PD}{\mathrm{PD}}
\newcommand{\DP}{\mathrm{DP}}

\def\BC{\mathrm{BC}}

\begin{document}

\title{\sc\bf\large\MakeUppercase{Poisson-Dirichlet approximation for the stationary distribution of the inclusion process
}}
\author{\sc Han~L.~Gan}
\date{\it University of Waikato}
\maketitle


\begin{abstract} 
We consider the approximation of the stationary distribution of the finite inclusion process with the Poisson-Dirichlet distribution. Using Stein's method, we derive an explicit bound for the approximation error, which is of order $1/N$ in the thermodynamic limit. The results are achieved from a minor modification to Stein's method for Poisson-Dirichlet distribution approximation developed in~\cite{GR21}. The derivatives used on test functions in~\cite{GR21} were directional type derivatives specifically chosen for their measure preserving properties. Depending upon the application, these derivatives can prove cumbersome. In this note, we show that for certain test functions we can instead use more traditional derivatives, which simplifies the bounds for the Stein factors and is more amenable to the approximation of the inclusion process.
\end{abstract}

\section{Introduction}
The one parameter Poisson-Dirichlet distribution with parameter $\theta > 0$, $\PD(\theta)$,  is a probability measure on the infinite dimensional ordered simplex,
\ba{
\nabla_\infty := \{ (p_1, p_2, \ldots): p_1 \geq p_2 \geq \ldots, \sum_{i=1}^\infty p_i = 1 \}.
}
The distribution lacks an explicit form, but it can be represented in a variety of ways, such as a limit of Dirichlet distributions or via a stick breaking process, see~\cite{Feng2010} for more details and examples. As working with order statistics can be technically challenging, rather than using non-increasing ordering,~\cite{GR21} use a labelled version of the process with labels from a compact metric space $E$, and define the Dirichlet process in the following manner.
\begin{definition}[Dirichlet process]
Let $\theta>0$ and let $\pi$ be a probability measure on a compact metric space~$E$.
Let $(P_1,P_2,\ldots)\sim \PD(\theta)$ be independent of 
$\xi_1,\xi_2,\ldots$, which are i.i.d.\ with distribution~$\pi$.
Define the probability measure on $M_1 := M_1(E)$, the space of probability measures on~$E$,  by 
\ben{\label{mucomp}
\DP(\theta, \pi) := \law\bbclr{\sum_{i=1}^\infty P_i \delta_{\xi_i}}.
}
In~\cite{GR21}, Stein's method for Dirichlet process approximation and Poisson-Dirichlet approximation was developed and a general theorem for bounding the approximation error between a random measure of interest and a Dirichlet process was derived. Motivated by the approximation of the stationary distribution of the inclusion process, we make a minor refinement to the framework of~\cite{GR21}. This refinement makes the framework more amenable to the application for approximation of the stationary distribution of the inclusion process, and moreover may make Poisson-Dirichlet and Dirichlet process approximation using Stein's method more accessible and easier to use in a further applications. 

\end{definition}
\subsection{Application to the inclusion process}
Originating in~\cite{GKR07} and \cite{GKRV09}, the discrete inclusion process is a stochastic interacting particle system with $N$ particles on $L$ sites, which we can index as $\{1, \ldots, L\}$. We define the set of configurations corresponding to the proportions of the $N$ particles on $L$ sites as
\ba{
\mathcal G_L = \left\{ \mu \in [0,1]^L: \sum_{i=1}^L \mu_i = 1 \right\},
}
where $\mu_i$ is the proportion of the $N$ particles in site $i$. As the indices for the sites are arbitrary, we can embed the atomic masses of the process onto a lattice of width $1/L$ on $[0,1]$. If we set $x_i = i/L$, for any $\theta > 0$, the generator of this discrete inclusion process is as follows,
\ben{
\cA_1 f(\mu) = \sum_{i,j=1}^L N\mu_i\left(N \mu_j + \frac{\theta}{L}\right) \left[ f\left(\mu - \frac1N \delta_{x_i} + \frac1N \delta_{x_j}\right) - f(\mu) \right]. \label{eq:discgen}
}
The process can be briefly described in the following manner. At unit rate, for each (ordered) pair of particles on sites $(x_i,x_j)$ a particle is moved from site $x_i$ to site $x_j$. In addition,  for each particle in the system, at rate $\theta/L$ independently, the particle is removed and a new particle is added at one of the $L$ sites chosen uniformly at random. Notably this model is also equivalent to a population genetics model, a type of Moran model with uniform mutation.

The limiting behaviour of this process has been extensively studied in a variety of different limiting regimes, see~\cite{CGG22},~\cite{JCG20},~\cite{CDG17} and~\cite{Kim21} for example. Let $W$ follow stationary distribution associated with $\cA_1$. In this paper, our goal is to assess the distribution of the masses in $W$, and compare this to distribution of the masses of the stationary distribution of the Dirichlet process with $E = [0,1]$ and $\pi = \U[0,1]$. We now state the main approximation result. Note that the definition of the family of test functions $\cH_2$ is formally defined in Section~\ref{sec:stn} and $\| \cdot \|_\infty$ denotes the supremum norm.
\begin{theorem}\label{mainthm}
For any $N$, $L$ and $\theta > 0$, let $W$ be distributed as the stationary probability measure associated with $\cA_1 f$ in~\eq{eq:discgen}. Set $E = [0,1]$, $\pi = \U[0,1]$ and let $Z\sim\DP(\theta,\pi)$. Then for any $h \in \cH_2$ such that $h(\mu) = \langle \phi, \mu^k \rangle$,
\ba{
\abs{ \E h(W) - \E h(Z)} \leq \bigg[ \frac{k(k-1)}{2L(\theta + 1)} + \frac{2 k(k-1)\theta}{N(\theta + 1)} + \frac{8k(k-1)(k-2)}{9 N (\theta + 2)} \bigg] \|\phi\|_\infty.
}
\end{theorem}
This bound is general, holds for any choice of $N, L$ and $\theta$, and will converge to zero as $N \to \infty$ when $L$ scales with $N$ in a positive manner. In~\cite[Section~3]{JCG20}, it is stated that if $d \rightarrow 0$ and $dL \rightarrow \alpha > 0$, then the limiting distribution is the Poisson-Dirichlet distribution. In the notation of this manuscript, $d = \theta / L$, and hence our bound is entirely consistent with existing results. Furthermore, in the so called thermodynamic limit, where $N, L \to \infty$ such that $N/L$ converges to a positive constant $\theta$, the bound is therefore of order $\frac{1}{N}$.

Moreover, using~\cite[Lemma 1.15]{GR21} we can show convergence in sampling formulas. Given a random sample of size $n$ from $W$ denoted by $(y_1, \ldots, y_n)$, let $\mathcal{S}_n(W)$ denote the denote the probability measure on set partitions of $\{ 1, \ldots, n\}$ induced by the relation $i \sim j \Leftrightarrow y_i = y_j$. Similarly, let $\mathcal{S}_n(Z)$ be analogous for $Z\sim \DP (\theta, \pi)$. Notably $\mathcal{S}_n(Z)$ is known to have an explicit form, namely the Ewens sampling formula. We then have the following result.
\begin{corollary}\label{corollary}
For any $N$, $L$ and $\theta>0$, let $W_N$ be distributed as the stationary probability measure of $\cA_1 f$ in~\eq{eq:discgen}. Then for $Z \sim \DP(\theta, \pi)$ and any $n > 0$,
\ba{
d_{TV}(\mathcal  S_n(W_N), \mathcal S_n(Z)) \leq \frac{n(n-1)}{2L(\theta + 1)} + \frac{2n(n-1)\theta}{N(\theta + 1)} + \frac{8n(n-1)(n-2)}{9 N (\theta + 2)}
}
\end{corollary}
To interpret the bound in the above corollary, suppose $L = \bigo(N^\alpha)$ for some $\alpha > 0$. Then the bound will be small when the sample size $n \ll  N^{\frac{\alpha}{2}}$ and $n \ll N^{\frac{1}{3}}$.

The remainder of the paper will consist of two sections. In the first section we will briefly outline Stein's method for Poisson-Dirichlet and Dirichlet process approximation and refine the bounds of~\cite{GR21} for the derivatives of the solutions to the Stein equation for test functions $h \in \cH_2$. In the final section we prove the approximation error bounds for the inclusion process.

\section{Refining the Stein method for the Dirichlet process}\label{sec:stn}
Pioneered in~\cite{Stein72} for Normal approximation, Stein's method is a powerful probabilistic tool used to explicitly bound the distributional distance between two distributions. Since the original seminal paper, it has been developed for a wide variety of distributions. For many examples and applications, see for example the surveys or monographs~\cite{Ross11, Chatterjee2014,introstein,LRS2017}.

We first give a brief summary of the Stein's method approach used in~\cite{GR21}. For any random probability measure $W$, and $Z \sim \DP(\theta, \pi)$, our goal is to bound quantities of the form $\abs{ \E h(W) - \E h(Z)}$, for all functions $h$ in a sufficiently rich class of test functions. To achieve this, we begin by characterising $\DP(\theta, \pi)$ as the unique stationary distribution of a Fleming-Viot process~\cite{FV79} with generator,
 \ban{\label{eq:FVgen}
\cA_2 f (\mu)=\frac{\theta}{2} \int_E\partial_{x} f(\mu) \bclr{\pi-\mu}(dx)
+ 
	 \frac12\int_{E^2}  \bclr{\mu(dx)\delta_x (dy) - \mu(dx) \mu(dy)} \partial_{xy} f(\mu).
}
The specific variant of Stein's method we use is often called the generator method, first introduced in~\cite{B88,B90} and~\cite{Gotze1991}. The general idea of the approach is that as $Z$ follows the stationary distribution of $\cA_2$, then $\E \cA_2 f(Z) = 0$ for all $f$. If for some random measure $W$ that is close to $Z$ in distribution, then one would expect $\E \cA_2 f(W) \approx 0$. The next step is then to choose specific test functions $f$ that quantify the distributional distance between $W$ and $Z$. More precisely, for any function $h$ from some family of functions $\cH$, we solve for the function $f_h$ that satisfies
\ban{\label{eq:stneq}
\cA_2 f_h(\mu) = h(\mu) - \E h(Z).
}
We then set $\mu = W$, which implies $\abs{\E h(W) - \E h(Z)} =  \abs{ \E \cA_2f_h(W)}$, and we then seek a bound for $\abs{\E \cA_2f_h(W)}$. 

Two different choices of families of test functions for $\cH$ are used in~\cite{GR21}.  Let $\mathrm{C}(E^K)$ denote the set of continuous functions from $E^k \mapsto \IR$. For $\mu \in M_1$ and $\phi \in \mathrm{C}(E)$, set $\langle \phi, \mu \rangle = \int_E \phi(x) \mu(dx)$ and $\BC^{2,1}(\IR^k)$ to be the set of functions with two bounded and continuous derivatives with second derivative Lipschitz. The first family of test functions is defined as,
\ba{
\cH_1 = \{ h(\mu) := H( \langle \phi_1, \mu \rangle, \ldots, \langle \phi_k, \mu \rangle) : k \in \IN, H \in \BC^{2,1}(\IR^k), \phi_i \in \mathrm{C}(E), i = 1, \ldots, k\}.
}
The second family of test functions is defined as,
\ba{
\cH_2 = \{ h(\mu) := \langle \phi, \mu^k \rangle : k \in \IN, \phi \in \mathrm{C}(E^k)\}.
}
To successfully apply Stein's method, bounds on $f_h$ in~\eq{eq:stneq} and its derivatives play a critical role. For each of the two families of test functions $\cH_1$ and $\cH_2$, two different approaches were used to bound the derivatives of $f_h$. For $\cH_1$, a probabilistic coupling approach that utilises properties of the coalescent process was used and for $\cH_2$ a different approach based upon a dual process was used. The differential operator for the derivatives used was defined by
\ban{
\partial_x F(\mu):= \lim_{\eps \to 0^+} \frac{ F( (1-\eps)\mu + \eps\delta_x) - F(\mu)}{\eps}, \label{eq:form1}
} 
as opposed to a more standard definition of derivative of
\ban{
\partial_x F(\mu) := \lim_{\eps \to 0^+} \frac{F(\mu + \eps \delta_x) - F(\mu)}{ \eps}. \label{eq:form2}
}
Derivatives of the form of~\eq{eq:form1} were used as for the family $\cH_1$, the probabilistic coupling arguments requires the measures to be evaluated by $F$ to be probability measures. For consistency, the same form of derivative was then also used for the family $\cH_2$ as well. It turns out that the proof method for $\cH_2$ does not require the measures to be probability measures, and as such, the more commonly form for derivatives as in~\eq{eq:form2} (see~\cite{EK93, EG93} for example) can be used and this is advantageous for our application to approximate of the stationary distribution of the inclusion process.

Let $W$ follow the stationary distribution of the discrete inclusion process associated with $\cA_1 f$ in~\eq{eq:discgen}. Then noting that this implies $\E \cA_1 f(W) = 0$, for any function $h \in \cH_2$, our approach will be to derive a bound for
\ban{
\abs{ \E h(W) - \E h (Z)} = \abs{ \E \cA_2 f_h(W)} = \abs{ \E \cA_2 f_h(W) - \E \cA_1 f_h(W)}, \label{eq:gencmp}
}
by exploiting properties of $f_h$ and its derivatives. This method is known as the generator comparison method in Stein's method literature. To evaluate the right hand side of the above, the natural approach is to take the Taylor expansion of $ f_h(\mu - \frac1N \delta_{i/L} + \frac1N \delta_{j/L}) - f_h(\mu)$ in $\cA_2 f$ to match the derivatives in $\cA_2f$. At this point, derivatives of the form~\eq{eq:form2} are more useful as they naturally match the terms in the Taylor expansion, whereas derivatives of the form~\eq{eq:form1} do not. Hence we rederive the results of~\cite[Theorem 2.7]{GR21} but for derivatives of the form~\eq{eq:form2}. We note that the proof method used in this section is effectively a simplification of the proof of~\cite[Theorem 2.7]{GR21} for test functions in $\cH_2$ using a different form of derivative, and as a result we attempt to only include the main ideas and details, and refer the reader to~\cite{GR21} for the full details.

For a test function $h \in \cH_2$, where $h(\mu) = \langle \phi, \mu^k \rangle$, elementary computations show that the generator $\cA_2$ can be written in the form
\ban{
\cA_2 h(\mu) = \sum_{1 \leq i < j \leq k} \big[ \langle \Phi^{(k)}_{ij} \phi, \mu^{k-1} \rangle - \langle \phi, \mu^k \rangle \big] + \frac{\theta}{2} \sum_{1 \leq i \leq k} \big[ \langle \phi, \mu^{i-1} \pi \mu^{k-i} \rangle - \langle \phi, \mu^k \rangle \big],\label{eq:dualgen}
}  
where $ \Phi^{(k)}_{ij} \phi (x_1, \ldots, x_{k-1}) = \phi(x_1, \ldots, x_{j-1}, x_i, x_j, \ldots, x_{k-1})$. Following~\cite[Section 3, page 353]{EK93}, we define a dual process $(Y_{\phi}(t))_{t \geq 0}$ on $\cup_{j \geq 0} \mathrm{C}(E^j)$ where we set $Y(0) = \phi$. Using the generator~\eq{eq:dualgen}, for each $i< j$, the process jumps from $\phi$ to $\Phi_{ij}^{(k)} \phi$ at rate 1, and at rate $\frac{\theta}{2}$, for each $1 \leq i \leq k$, $\phi$ transitions to a $k-1$ dimensional function by integrating dimension $i$ with respect to $\pi$. Notably this process has similarities to the inclusion process, hence the approximation result in Theorem~\ref{mainthm} is unsurprising. Using standard duality arguments~\cite[Theorem 3.1]{EK93}, setting $Z_\mu(t)$ to be the Fleming-Viot processes governed by $\cA_2$ with $Z_\mu(0) = \mu$, 
\ban{
\E \big[ h(Z_\mu(t)) \big] = \E \big[ \langle Y_{\phi}(t), \mu^{k-M_k(t)} \rangle  \ind[ M_k(t) < k]\big], \label{eq:dual}
}
where $M_k(t)$ is the number of transitions of the process $Y_\phi(\cdot)$ up to time $t$. Note that we can without loss of generality we assume that $\E h(Z) = 0$.
\begin{theorem}\label{thm:stnfcts}
For $h \in \cH_2$ such that $h(\mu) = \langle \phi,\mu^k\rangle$, the solution $f_h$ given in~\eq{eq:stneq} is well defined and
\ba{
\sup_x\| \partial_x f_h \|_\infty &\leq \frac{2k}{\theta} \|\phi\|_\infty, \\
\sup_{x,y} \| \partial_{xy} f_h\|_\infty &\leq \frac{k(k-1)}{\theta+1} \|\phi\|_\infty,\\
\sup_{x,y,z} \| \partial_{xyz} f_h\|_\infty & \leq \frac{2k(k-1)(k-2)}{3(\theta+2)} \|\phi\|_\infty.
}
\end{theorem}
\begin{proof}
We begin by noting that noting that~\cite{GR21} already show that $f_h$ is well defined and satisfies~\eq{eq:stneq}. In particular the solution is given by
\ba{
f_h(\mu) &= - \int_0^\infty  [\E h(Z_\mu(t)) - \E h (Z)]  dt\\
	&= -\int_0^\infty \E \big[ \langle Y_\phi(t), \mu^{k-M(t)} \rangle  \ind[ M_k(t) < k] \big] dt,
} 
where $(Z_\mu(t))_{t \geq 0}$ is a Fleming-Viot process following generator $\cA_2$ and $Z_\mu(0) = \mu$ in the first line and the second line follows from using the duality~\eq{eq:dual} and the assumption that $\E h(Z) = 0$. For the first derivative,
\ba{
\partial_x f_h(\mu) &= - \lim_{\eps \to 0}\frac{1}{\eps} \int_0^\infty \E\bigg[ \big[ \langle Y_\phi(t), (\mu + \eps \delta_x)^{k - M(t)} \rangle - \langle Y_\phi(t), \mu^{k-M(t)} \rangle \big]  \ind[ M_k(t) < k] \big] \bigg] dt\\
	&= - \lim_{\eps \to 0} \frac{1}{\eps} \int_0^\infty \sum_{i=1}^k \E \int_{\tau_{i-1}}^{\tau_i} \big[ \langle Y_\phi(t), (\mu + \eps \delta_x)^{k - i+1)} \rangle - \langle Y_\phi(t), \mu^{k-i+1} \rangle \big] dt,
}
where $\tau_i$ denotes the $i$-th transition for for the process $Y_\phi(t)$ and $\tau_0 = 0$. Noting that 
\[ \|  \langle Y_\phi(t), (\mu + \eps \delta_x)^{s)} \rangle - \langle Y_\phi(t), \mu^{s} \rangle\| \leq \eps s \|\phi\|_\infty + \lito (\eps), \]
 and $\tau_{i} - \tau_{i-1}$ is exponentially distributed with rate $(k-i+1)(k-i+\theta)/2$, yields
\ba{
\sup_x \|\partial_x f_h\|_\infty &\leq \lim_{\eps \to 0}\frac{1}{\eps} \bigg[\sum_{i=1}^k \left(\|\phi\|_\infty  (k-i+1)  \eps \cdot \E(\tau_i - \tau_{i-1}) + \lito(\eps)\right) \bigg]\\
	&= \|\phi\|_\infty \sum_{i=1}^k \frac{2}{k-i+\theta}\\
	&\leq \frac{2k}{\theta}  \|\phi\|_\infty.
}
For the second and third derivatives, using the same arguments and that 
\ba{ \|  \langle Y_\phi(t), (\mu + \eps_1 \delta_x + \eps_2\delta_y)^{s}& \rangle - \langle Y_\phi(t), (\mu + \eps_1 \delta_x)^{s} \rangle - \langle Y_\phi(t), (\mu + \eps_2\delta_y)^{s} \rangle + \langle Y_\phi(t), \mu^{s} \rangle\| \\
&\leq \eps_1\eps_2 s(s-1) \|\phi\|_\infty + \lito (\eps_1\eps_2) 
}
for the second derivative and analogously for the third derivative yields the following bounds.
\ba{
\sup_{x,y} \| \partial_{xy} f_h\|_\infty &\leq \sum_{i=1}^{k-1} \| \phi\|_\infty (k-i+1)(k-i)\E(\tau_i - \tau_{i-1})\\
	&= \|\phi\|_\infty \sum_{i=1}^{k-1} \frac{2(k-i)}{k-i+\theta}\\
	&\leq \frac{k(k-1)}{\theta+1} \|\phi\|_\infty.
}
\ba{
\sup_{x,y,z} \| \partial_{xyz} f_h\|_\infty &\leq \sum_{i=1}^{k-2} \| \phi\|_\infty (k-i+1)(k-i)(k-i-1)\E(\tau_i - \tau_{i-1})\\
	&= \|\phi\|_\infty \sum_{i=1}^{k-2} \frac{2(k-i)(k-i-1)}{k-i+\theta}\\
	&\leq \frac{2k(k-1)(k-2)}{3(\theta+2)} \|\phi\|_\infty.
}

\end{proof}

\section{Proof of the inclusion process approximation results}
\begin{proof}[Proof of Theorem~\ref{mainthm}]
Recall the generator we use for the Fleming-Viot process~\eq{eq:FVgen} with $E = [0,1]$ and $\pi$ the uniform measure on $E$ is 
\ben{\label{eq:dpgen}
\cA_2 f (\mu)=\theta \int_0^1 \partial_{x} f(\mu) \bclr{\pi-\mu}(dx)
+ 
	 \int_0^1 \int_0^1 \bclr{\mu(dx)\delta_x (dy) - \mu(dx) \mu(dy)} \partial_{xy} f(\mu),
}
and that the generator for the (discrete) inclusion process is
\ba{
\cA_1 f(\mu) = \sum_{i,j=1}^L N \mu_i\left(N \mu_j + \frac{\theta}{L}\right) \left[ f\left(\mu - \frac1N \delta_{i/L} + \frac1N \delta_{j/L}\right) - f(\mu) \right]. 
}
Furthermore, recall from~\eq{eq:gencmp} that our goal is to bound
\ban{
\abs{ \E h(W) - \E h (Z)} = \abs{ \E \cA_2 f_h(W)} = \abs{ \E \cA_2 f_h(W) - \E \cA_1 f_h(W)}.\label{eq:1}
}
where $W$ to be a random measure following the stationary distribution associated with $\cA_1$. For convenience we write $W = \sum_{i=1}^L W_i \delta_{x_i}$, where $x_i = i/L$. 
Our goal is to use the Taylor expansion $\cA_1f(W)$, match terms to $\cA_2f(W)$ and then carefully bound the remainder. We note the Taylor expansion,
\ba{
f \left(W - \frac1N \delta_{x_i} + \frac1N \delta_{x_j}\right) - f(W) &= -\frac{1}{N} \partial_{x_i} f(W) + \frac1N \partial_{x_j} f(W) \\
	&\ \  + \frac{1}{2N^2} \partial^2_{x_ix_i} f(W)+ \frac{1}{2N^2} \partial^2_{x_jx_j} f(W) - \frac{1}{N^2} \partial^2_{x_ix_j} f(W)\\
	&\ \ +R(W).
}
Then given we can sum either over all $i, j$ or $i \neq j$ when mathematically convenient as for $i = j$, $f \left(W - \frac1N \delta_{x_i} + \frac1N \delta_{x_j}\right) - f(W)= 0$,
\ban{
\cA_1f(W) &= \sum_{i \neq j}^LN^2  W_iW_j \bigg[-\frac{1}{N} \partial_{x_i} f(W) + \frac1N \partial_{x_j} f(W)  \notag \\
	&\ \ + \frac{1}{2N^2} \partial^2_{x_ix_i} f(W)+ \frac{1}{2N^2} \partial^2_{x_jx_j} f(W) - \frac{1}{N^2} \partial^2_{x_ix_j} f(W) + \eps_3(W)\bigg]\notag \\
	&\ \ +  \frac{\theta}{L} \sum_{i,j=1}^L NW_i \bigg[ -\frac1N \partial_{x_i}f(W) + \frac1N \partial_{x_j}f(W) + \eps_2(W) \bigg]\notag \\
	&= \sum_{i=1}^L W_i(1-W_i) \partial^2_{x_ix_i} f(W) -  \sum_{i \neq j}^L W_iW_j \partial_{x_i x_j}f(W) + N^2\sum_{i \neq j}^L W_iW_j\eps_3(W)\notag \\
	&\ \ - \theta\sum_{i =1}^L W_i \partial_{x_i}f(W) + \frac{\theta}{L} \sum_{j=1}^L \partial_{x_j}f(W) + \frac{N\theta}{L} \sum_{i,j=1}^L W_i \eps_2(W),\label{eq:2}
}
where $\eps_2(W)$ and $\eps_3(W)$ are higher order expansion terms that we currently suppress for readability but will explicitly bound later. Again using the representation $W = \sum_{i=1}^L W_i \delta x_i$, we can write
\ban{
\cA_2 f(W) &= \theta \int_0^1 \partial_x f(W) dx - \theta \sum_{i=1}^L W_i \partial_{x_i}f(W) \notag \\
	&\hspace{1cm}+ \sum_{i=1}^L W_i(1-W_i) \partial^2_{x_i x_i} f(W) - \sum_{i \neq j}^LW_iW_j \partial^2_{x_ix_j} f(W). \label{eq:3}
}
Combining~\eq{eq:1}, \eq{eq:2} and \eq{eq:3} yields
\ba{
\abs{ \E h(W) - \E h(Z) } \leq N\theta\abs{ \eps_2(W)} + N^2\abs{\eps_3(W)} + \theta \bbbabs{ \int_0^1 \partial_xf_h(W) dx - \sum_{j=1}^L \frac{1}{L} \partial_{x_j}f_h(W) }.
}
For the final term of the above, note that this is simply the Riemann sum approximation for the integral, and hence via the mean value theorem and Theorem~\ref{thm:stnfcts},
\ba{
\bbbabs{ \int_0^1 \partial_xf_h(W) dx - \sum_{j=1}^L \frac{1}{L} \partial_{x_j}f_h(W) } \leq \frac{\sup_{x,y} \| \partial_{xy}f_h\|_\infty }{2L} \leq \frac{k(k-1)}{2L(\theta+1)} \|\phi\|_\infty.
}
The final result follows from bounding the higher order error terms as below.
\ba{
\eps_2(W) =\frac{1}{2N^2} \partial^2_{x_ix_i} f_h(W^*)+ \frac{1}{2N^2} \partial^2_{x_jx_j} f_h(W^*) - \frac{1}{N^2} \partial^2_{x_ix_j} f_h(W^*),
}
where $W^* = (W + s_1( -\frac{1}{N}\delta_{x_i} + \frac{1}{N} \delta_{x_j}))$ for some $s_1 \in [0,1]$. Hence,
\ba{
\abs{\eps_2(W)} \leq \frac{2}{N^2} \sup_{x,y}\|\partial_{xy} f_h\|_\infty \leq \frac{2k(k-1)}{N^2(\theta + 1)} \|\phi\|_\infty.
}
Similarly,
\ba{
\eps_3(W) = -\frac{1}{6N^3} \partial^3_{x_ix_ix_i}f_h(W^*) + \frac{1}{2N^3} \partial^3_{x_ix_ix_j}f_h(W^*) - \frac{1}{2N^3} \partial^3_{x_ix_ix_j}f_h(W^*) + \frac{1}{6N^3} \partial^3_{x_jx_jx_j}f_h(W^*),
}
where $W^* = (W + s_2( -\frac{1}{N}\delta_{x_i} + \frac{1}{N} \delta_{x_j}))$ for some $s_2 \in [0,1]$. Hence
\ba{
\abs{\eps_3(W)} \leq \frac{4}{3N^3} \sup_{x,y,z}\|\partial_{xyz} f_h\|_\infty \leq \frac{8k(k-1)(k-2)}{9N^3(\theta + 2)} \|\phi\|_\infty.
}
\end{proof}
\begin{proof}[Proof of Corollary~\ref{corollary}]
We omit the details of this proof as it is exactly the same as the proof of \cite[Lemma~1.15]{GR21}, but using the derivative bounds of Theorem~\ref{thm:stnfcts} instead of the bounds of \cite[Theorem~2.7]{GR21}. 
\end{proof}
\bibliographystyle{apalike}
\bibliography{references}

\end{document}